\documentclass[10pt,final,twoside]{article}

\usepackage{amsmath,amsthm,amssymb}
\numberwithin{equation}{section}

\newcommand{\reftab}{\mathbb{T}}

\def\N{\mathbb{N}}
\def\C{\mathbb{C}}

\def\sym#1{\mathfrak{S}_{#1}}

\def\card#1{\left|#1\right|}
\def\kakko#1{\left(#1\right)}
\def\ckakko#1{\left\{#1\right\}}

\def\ceil#1{\left\lceil#1\right\rceil}

\newcommand{\inprod}[3][{}]{\left\langle#2,\,#3\right\rangle_{\!#1}}

\def\st{\operatorname{Sym}}

\def\len#1{\ell(#1)}
\def\deq{\overset{\text{def}}{=}}

\def\triv#1{\boldsymbol{1}_{#1}}

\def\bc{\overline{c}}

\def\phi{\varphi}

\newcommand{\tramat}[2][n,l]
{\boldsymbol{F}^{#2}_{\tmspace -\thinmuskip{.3em}#1}}
\newcommand{\gcp}[3][\alpha]{F^{#2}_{#3}(#1)}

\newcommand{\cp}[1]{f_{#1}}

\def\liegl{\mathfrak{gl}}

\newcommand{\U}[1][n]{\mathcal{U}(\liegl_{#1})}

\def\ys#1{c_{\lower0.3ex\hbox{$\scriptstyle#1$}}}
\def\bys#1{\bc_{\lower0.3ex\hbox{$\scriptstyle#1$}}}

\DeclareMathOperator{\tr}{tr}

\DeclareMathOperator{\ind}{ind}

\DeclareMathOperator{\End}{End}
\DeclareMathOperator{\Mat}{Mat}
\DeclareMathOperator{\sgn}{sgn}
\DeclareMathOperator{\rank}{rk}
\DeclareMathOperator{\per}{per}
\DeclareMathOperator{\fix}{fix}

\newcommand{\adet}[1][\alpha]%
{\operatorname{det}^{(#1)}}
\newcommand{\cw}[1]{\nu(#1)} 
\newcommand{\anu}[1]{\alpha^{\cw{#1}}}
\newcommand{\ve}{\boldsymbol{e}}

\newcommand{\A}[1][n]{\mathcal{P}(\Mat_{#1})}
\newcommand{\GLmod}[2]{\mathcal{M}_{#1}^{#2}}
\newcommand{\Smod}[1]{\mathcal{S}^{#1}}
\newcommand{\Cmod}[2]{\boldsymbol{V}_{\!\!#1,#2}}

\newcommand{\sfphi}[2][K]{\phi^\sharp_{\!\lower0.4ex\hbox{$\scriptstyle#1,#2$}}}

\newcommand{\intertwiner}{\mathcal{T}}

\newcommand{\tHGF}[5]
{{}_{#1}\tilde{F}_{#2}\kakko{\genfrac{}{}{0pt}{}{#3}{#4};#5}}
\def\phs#1#2{(#1)_{#2}}

\newtheorem{thm}{Theorem}[section]
\newtheorem{lem}[thm]{Lemma}

\theoremstyle{definition}

\newtheorem{prob}[thm]{Problem}
\newtheorem{ex}[thm]{Example}

\theoremstyle{remark}
\newtheorem{rem}[thm]{Remark}

\newenvironment{keywords}{\smallskip\noindent{\bfseries Keywords:}}{}
\newenvironment{MSC}{\smallskip\noindent{\bfseries 2000 Mathematical Subject Classification:}}{}
\begin{document}
\title{\bfseries Representation theory of the $\alpha$-determinant\\
and zonal spherical functions}
\author{Kazufumi KIMOTO}
\date{December 14, 2007}
\maketitle

\begin{abstract}
We prove that the multiplicity of each irreducible component
in the $\mathcal{U}(\mathfrak{gl}_n)$-cyclic module
generated by the $l$-th power $\det^{(\alpha)}(X)^l$
of the $\alpha$-determinant
is given by the rank of a matrix whose entries are
given by a variation of
the spherical Fourier transformation for $(\mathfrak{S}_{nl},\mathfrak{S}_l^n)$.
Further, we calculate the matrix explicitly when $n=2$.
This gives not only another proof of the result by Kimoto-Matsumoto-Wakayama (2007)
but also a new aspect of the representation theory of the $\alpha$-determinants.


\begin{MSC}
22E47, 
43A90. 
\end{MSC}
\end{abstract}

\section{Introduction}

Let $n$ be a positive integer.
We denote by $\sym n$ the symmetric group of degree $n$.
For a permutation $\sigma\in\sym n$, we define
$$
\cw\sigma\deq\sum_{i\ge1}(i-1)m_i(\sigma)\qquad(\sigma\in\sym n),
$$
where $m_i(\sigma)$ is the number of $i$-cycles
in the disjoint cycle decomposition of $\sigma$.
We notice that $\cw\cdot$ is a class function on $\sym n$.
It is easy to see that $(-1)^{\cw\sigma}=\sgn\sigma$ is the signature
of a permutation $\sigma$.

Let $\alpha$ be a complex number
and $A=(a_{ij})_{1\le i,j\le n}$ an $n$ by $n$ matrix.
The $\alpha$-determinant $\adet(A)$ of $A$ is defined by
\begin{equation}\label{eq:def_of_adet}
\adet(A)\deq\sum_{\sigma\in\sym n}\anu\sigma a_{\sigma(1)1}a_{\sigma(2)2}\dots a_{\sigma(n)n}.
\end{equation}
We readily see that
the $\alpha$-determinant
$\adet(A)$ coincides with the determinant $\det(A)$
(resp. permanent $\per(A)$) of $A$
when $\alpha=-1$ (resp. $\alpha=1$).
Hence we regard
the $\alpha$-determinant as a common generalization
of the determinant and permanent.

The $\alpha$-determinant is first introduced by Vere-Jones \cite{VereJones}.
He proved the identity
\begin{equation}\label{eq:VJ}
\det(I-\alpha A)^{-1/\alpha}
=\sum_{k=0}^\infty \frac1{k!}
\sum_{1\le i_1,\dots,i_k\le n}
\adet\!\!\begin{pmatrix}
a_{i_1i_1}&\dots&a_{i_1i_k}\\
\vdots&\ddots&\vdots\\
a_{i_ki_1}&\dots&a_{i_ki_k}
\end{pmatrix}
\end{equation}
for an $n$ by $n$ matrix $A=(a_{ij})_{1\le i,j\le n}$
such that the absolute value of any eigenvalue of $A$ is less than $1$.
Here $I$ denotes the identity matrix of suitable size.
His intention of the study of the $\alpha$-determinant is
an application to probability theory.
Actually,
the identity \eqref{eq:VJ} supplies a unified treatment
of the multivariate binomial and negative binomial distributions.
Further, Shirai and Takahashi \cite{ST} proved
a Fredholm determinant version of \eqref{eq:VJ}
for a trace class integral operator and use it
to define a certain one-parameter family
of point processes.
We note that
a pfaffian analogue of the Vere-Jones identity \eqref{eq:VJ}
has been also established and is applied to probability theory
by Matsumoto \cite{M2005}.
It is also worth noting that \eqref{eq:VJ} is obtained by
specializing $p_i(x)=\alpha^{i-1}$ and regarding $y_1, \dots, y_n$
as eigenvalues of $A$ in the Cauchy identity
\begin{equation}\label{eq:Cauchy}
\prod_{i,j\ge1}\frac1{1-x_iy_j}
=\sum_{\lambda}\frac{1}{z_\lambda}p_\lambda(x)p_\lambda(y),
\end{equation}
where $\lambda$ in the right-hand side runs over the set of all partitions,
$z_\lambda$ denotes the cardinality
of the centralizer of a permutation whose cycle type is $\lambda$,
and $p_\lambda$ denotes the power-sum symmetric function
corresponding to $\lambda$
(see \cite{M} for detailed information on symmetric functions).
In fact, under the specialization,
the left-hand side of \eqref{eq:Cauchy} becomes $\det(1-\alpha t)^{-1/\alpha}$
and the right-hand side represents its expansion
in terms of $\alpha$-determinants (see also \cite{K2007c}).

In this article,
we focus our attention on the representation-theoretic aspect of
the $\alpha$-determinant.
Let $\U$ be the universal enveloping algebra
of the general linear Lie algebra $\liegl_n=\liegl_n(\C)$,
and $\A$ be the polynomial algebra in the $n^2$ variable $x_{ij}$
$(1\le i,j\le n)$.
We put $X=(x_{ij})_{1\le i,j\le n}$ and write an element in $\A$
as $f(X)$ in short.
The algebra $\A$ becomes a left $\U$-module via
\begin{align*}
E_{ij}\cdot f(X)=\sum_{s=1}^n x_{is}\frac{\partial f(X)}{\partial x_{js}}
\end{align*}
for $f(X)\in\A$
where $\{E_{ij}\}_{1\le i,j\le n}$ is the standard basis of $\liegl_n$.
Now we regard
the $\alpha$-determinant $\adet(X)$ of $X$
as an element in $\A$
and consider the cyclic submodule
\begin{align*}
\Cmod nl(\alpha)\deq\U\cdot\adet(X)^l
\end{align*}
of $\A$. Since
\begin{equation}
\Cmod n1(-1)=\U\cdot\det(X)\cong\GLmod n{(1^n)},\qquad
\Cmod n1(1)=\U\cdot\det(X)\cong\GLmod n{(n)},
\end{equation}
the module $\Cmod n1(\alpha)$ is regarded as an
interpolation of these two irreducible representations.
Here we denote by $\GLmod n\lambda$
the irreducible $\U$-module whose highest weight is $\lambda$.
We notice that
we can identify the dominant integral weights with partitions
as far as we consider the polynomial representations of $\U$.

Our main concern is to solve the
\begin{prob}\label{prob:1}
Describe the irreducible decomposition of
the $\U$-module $\Cmod nl(\alpha)$ explicitly.
\end{prob}

In \cite{KMW},
the following general result on $\Cmod nl(\alpha)$ is proved.
\begin{thm}\label{thm:main}
For each $\lambda\vdash nl$ such that $\len\lambda\le n$,
there exists a certain square matrix
$\tramat\lambda(\alpha)$ of size $K_{\lambda(l^n)}$
whose entries are polynomials in $\alpha$ such that
\begin{align*}
\Cmod nl(\alpha)\cong\bigoplus_{\substack{\lambda\vdash nl\\\len\lambda\le n}}
(\GLmod n\lambda)^{\oplus\rank\tramat\lambda(\alpha)}.
\end{align*}
Here $K_{\lambda\mu}$ denotes the Kostka number
and $\len\lambda$ is the length of $\lambda$.
\end{thm}

We call this matrix $\tramat\lambda(\alpha)$
the \emph{transition matrix} for $\lambda$ in $\Cmod nl(\alpha)$.
We notice that the transition matrix is determined up to conjugacy.
Thus, Problem \ref{prob:1} is reduced to
the determination of the matrices $\tramat\lambda(\alpha)$
relative to a certain (nicely chosen) basis.
Up to the present,
we have obtained an explicit form of $\tramat\lambda(\alpha)$
in only several particular cases.

\begin{ex}\label{ex:MW2005}
When $l=1$, Problem \ref{prob:1} is completely solved in \cite{MW2005} as follows:
For each positive integer $n$, we have
\begin{equation}
\Cmod n1(\alpha)=
\U\cdot\adet(X)
\cong\bigoplus_{\substack{\lambda\vdash n\\ \cp\lambda(\alpha)\ne0}}
\kakko{\GLmod n\lambda}^{\oplus f^\lambda},
\end{equation}
where $f_\lambda(\alpha)$ is a (modified) \emph{content polynomial}
\begin{align*}
f_\lambda(\alpha)\deq\prod_{i=1}^{\len\lambda}\prod_{j=1}^{\lambda_i}
(1+(j-i)\alpha).
\end{align*}
In other words, for each $\lambda\vdash n$, we have
\begin{equation}
\text{multiplicity of $\GLmod n\lambda$ in $\Cmod n1(\alpha)$}
=\begin{cases}
0 & \alpha\in\ckakko{1/k\,;\,1\le k<\len\lambda}
\sqcup\ckakko{-1/k\,;\,1\le k<\lambda_1},\\
f^\lambda & \text{otherwise.}
\end{cases}
\end{equation}
The transition matrix $\tramat[n,1]\lambda(\alpha)$ in this case
is given by $f_\lambda(\alpha)I$.
\end{ex}

\begin{ex}\label{ex:Jacobi}
When $n=2$, the transition matrix $\tramat[2,l]\lambda(\alpha)$
is of size $1$ (i.e. just a polynomial) and
it is shown in \cite[Theorem 4.1]{KMW} that
\begin{equation}\label{eq:decomp_of_V2l}
\Cmod 2l(\alpha)=
\U[2]\cdot\adet(X)^l
\cong\bigoplus_{\substack{0\le s\le l\\ F_{2,l}^{(2l-s,s)}(\alpha)\ne0}}
\GLmod 2{(2l-s,s)},
\end{equation}
where we put
\begin{align*}
F_{2,l}^{(2l-s,s)}(\alpha)
&=(1+\alpha)^{l-s}G^l_s(\alpha),\\
G^l_s(\alpha)
&=\sum_{j=0}^l\frac{\phs{-s}j\phs{l-s+1}j}{\phs{-l}j}\frac{(-\alpha)^j}{j!}.
\end{align*}
Here $\phs aj=\Gamma(a+j)/\Gamma(a)$ is the Pochhammer symbol.
We note that $G^l_s(\alpha)$ is written by a \emph{Jacobi polynomial} as
\begin{align*}
G^l_s(\alpha)=\binom{s-l-1}s^{\!\!-1}
P^{(-l-1,2l-2s+1)}(1+2\alpha).
\end{align*}
\end{ex}


In this paper,
we show that the entries of the transition matrices  $\tramat\lambda(\alpha)$
are given by a variation of
the spherical Fourier transformation
of a certain class function on $\sym{nl}$
with respect to the subgroup $\sym l^n$ (Theorem \ref{thm:mymain}).
This result also provides another proof of Theorem \ref{thm:main}.
Further, we give a new calculation of
the polynomial $F_{2,l}^{(2l-s,s)}(\alpha)$ in Example \ref{ex:Jacobi}
by using an explicit formula for the values of zonal spherical functions
for the Gelfand pair $(\sym{2n},\sym n\times\sym n)$ due to Bannai and Ito
(Theorem \ref{thm:my_n=2_case}).

\section{Irreducible decomposition of $\Cmod nl(\alpha)$}

Fix $n,l\in\N$.
Consider the standard tableau $\reftab$ with shape $(l^n)$
such that the $(i,j)$-entry of $\reftab$ is $(i-1)l+j$.
For instance, if $n=3$ and $l=2$, then
\begin{equation*}
\reftab=\lower1.5em\hbox{\setlength{\unitlength}{1.15em}%
\begin{picture}(2,3)
\multiput(0,0)(0,1){4}{\line(1,0){2}}
\multiput(0,0)(1,0){3}{\line(0,1){3}}
\put(0,2){\makebox(1,1){$1$}}
\put(1,2){\makebox(1,1){$2$}}
\put(0,1){\makebox(1,1){$3$}}
\put(1,1){\makebox(1,1){$4$}}
\put(0,0){\makebox(1,1){$5$}}
\put(1,0){\makebox(1,1){$6$}}
\end{picture}}\,.
\end{equation*}
We denote by $K=R(\reftab)$ and $H=C(\reftab)$
the row group and column group of the standard tableau $\reftab$
respectively.
Namely,
\begin{equation}
K=\ckakko{g\in\sym{nl}\,;\,\ceil{g(x)}=\ceil{x},\ x\in[nl]},\quad
H=\ckakko{g\in\sym{nl}\,;\,g(x)\equiv x\pmod l,\ x\in[nl]},
\end{equation}
where we denote by $[nl]$ the set $\{1,2,\dots,nl\}$.
We put
\begin{equation}
e=\frac1{\card K}\sum_{k\in K}k\in\C[\sym{nl}].
\end{equation}
This is clearly an idempotent element in $\C[\sym{nl}]$.
Let $\varphi$ be a class function on $H$.
We put
\begin{align*}
\Phi\deq\sum_{h\in H}\varphi(h)h\in\C[\sym{nl}].
\end{align*}
Consider the tensor product space $V=(\C^n)^{\otimes nl}$.
We notice that $V$ has a $(\U,\C[\sym{nl}])$-module structure given by
\begin{align*}
E_{ij}\cdot \ve_{i_1}\otimes\dots\otimes\ve_{i_{nl}}
&\deq
\sum_{s=1}^{nl}\delta_{i_{s},j}\,
\ve_{i_{1}}\otimes\dots\otimes\overset{\text{$s$-th}}{\ve_i}\otimes\dots\otimes\ve_{i_{nl}},\\
\ve_{i_1}\otimes\dots\otimes\ve_{i_{nl}}\cdot\sigma
&\deq
\ve_{i_{\sigma(1)}}\otimes\dots\otimes\ve_{i_{\sigma(nl)}}
\qquad(\sigma\in\sym{nl})
\end{align*}
where $\{\ve_i\}_{i=1}^{n}$ denotes the standard basis of $\C^n$.
The main concern of this section is to solve the
\begin{prob}\label{prob:generalized}
Describe the irreducible decomposition of
the left $\U$-module $V\cdot e\Phi e$.
\end{prob}


Here we show that Problem \ref{prob:generalized}
includes Problem \ref{prob:1} as a special case.
We consider the group isomorphism
$\theta:H\to\sym n^l$ defined by
\begin{align*}
\theta(h)\deq(\theta(h)_1,\dots,\theta(h)_l);\quad
\theta(h)_i(x)=y \iff h((x-1)l+i)=(y-1)l+i.
\end{align*}
We also define an element $D(X;\phi)\in\A$ by
\begin{align*}
D(X;\phi)&\deq
\sum_{h\in H}\phi(h)
\prod_{q=1}^n\prod_{p=1}^l x_{\theta(h)_p(q),q}
=\sum_{h\in H}\phi(h)
\prod_{q=1}^n\prod_{p=1}^l x_{q,\theta(h)^{-1}_p(q)}\\
&=\sum_{\sigma_1,\dots,\sigma_l\in\sym n}
\phi(\theta^{-1}(\sigma_1,\dots,\sigma_l))
\prod_{q=1}^n\prod_{p=1}^l x_{\sigma_p(q),q}.
\end{align*}
We note that $D(X;\anu\cdot)=\adet(X)^l$
since
$\cw{\theta^{-1}(\sigma_1,\dots,\sigma_l)}
=\cw{\sigma_1}+\dots+\cw{\sigma_l}$
for $(\sigma_1,\dots,\sigma_l)\in\sym l^n$.

Take a class function $\delta_H$ on $H$ defined by
\begin{align*}
\delta_H(h)=\begin{cases}
1 & h=1 \\ 0 & h\ne1.
\end{cases}
\end{align*}
We see that $D(X;\delta_H)=(x_{11}x_{22}\dots x_{nn})^l$.
We need the following lemma
(see {\cite[Lemma 2.1]{KMW}} for the proof of (1).
The assertion (2) is immediate).

\begin{lem}\label{lem:cyclic}
{\upshape(1)}
It holds that
\begin{align*}
&\U\cdot\ve_1^{\otimes l}\otimes\dots\otimes\ve_n^{\otimes l}
=V\cdot e=\st^l(\C^n)^{\otimes n},\\
&\U\cdot D(X;\delta_H)
=\bigoplus_{\substack{i_{pq}\in\{1,2,\dots,n\}\\(1\le p\le l,\,1\le q\le n)}}
\C\cdot\prod_{q=1}^n\prod_{p=1}^l x_{i_{pq}q}
\cong\st^l(\C^n)^{\otimes n}.
\end{align*}

\noindent{\upshape(2)}
The map
\begin{align*}
\intertwiner:\U\cdot D(X;\delta_H)\ni
\prod_{q=1}^n\prod_{p=1}^l x_{i_{pq}q}
\longmapsto
(\ve_{i_{11}}\otimes\dots\otimes\ve_{i_{l1}})\otimes\dots\otimes%
(\ve_{i_{1n}}\otimes\dots\otimes\ve_{i_{ln}})\cdot e\in V\cdot e
\end{align*}
is a bijective $\U$-intertwiner.
\qed
\end{lem}

We see that
\begin{align*}
\intertwiner\kakko{D(X;\phi)}
&=\sum_{h\in H}\phi(h)
\intertwiner\kakko{\prod_{q=1}^n\prod_{p=1}^l x_{\theta(h)_p(q),q}}\\
&=\sum_{h\in H}\phi(h)
(\ve_{\theta(h)_1(1)}\otimes\dots\otimes\ve_{\theta(h)_l(1)})\otimes\dots\otimes%
(\ve_{\theta(h)_1(n)}\otimes\dots\otimes\ve_{\theta(h)_l(n)})\cdot e\\
&=\ve_1^{\otimes l}\otimes\dots\otimes\ve_n^{\otimes l}%
\cdot\sum_{h\in H}\phi(h)h\cdot e
=\ve_1^{\otimes l}\otimes\dots\otimes\ve_n^{\otimes l}%
\cdot e\Phi e
\end{align*}
by (2) in Lemma \ref{lem:cyclic}.
Using (1) in Lemma \ref{lem:cyclic},
we have the
\begin{lem}
It holds that
\begin{align*}
\U\cdot D(X;\phi)\cong V\cdot e\Phi e
\end{align*}
as a left $\U$-module.
In particular, $V\cdot e\Phi e\cong\Cmod nl(\alpha)$
if $\varphi(h)=\anu h$.
\qed
\end{lem}



By the Schur-Weyl duality, we have
\begin{align*}
V\cong\bigoplus_{\lambda\vdash nl}\GLmod n\lambda\boxtimes\Smod\lambda.
\end{align*}
Here
$\Smod\lambda$ denotes the irreducible unitary right $\sym{nl}$-module
corresponding to $\lambda$.
We see that
\begin{align*}
\dim\kakko{\Smod\lambda\cdot e}
=\inprod[\sym{nl}]{\ind_K^G\triv K}{\Smod\lambda}
=K_{\lambda(l^n)}
\end{align*}
where $\triv K$ is the trivial representation of $K$ and
$\inprod[\sym{nl}]\pi\rho$ is the intertwining number
of given representations $\pi$ and $\rho$ of $\sym{nl}$.
Since $K_{\lambda(l^n)}=0$ unless $\len\lambda\le n$,
it follows the
\begin{thm}
It holds that
\begin{align*}
V\cdot e\Phi e
\cong\bigoplus_{\substack{\lambda\vdash nl\\ \len\lambda\le n}}
\GLmod n\lambda\boxtimes\kakko{\Smod\lambda\cdot e\Phi e}.
\end{align*}
In particular, as a left $\U$-module,
the multiplicity of $\GLmod n\lambda$ in $V\cdot e\Phi e$
is given by
\begin{align*}
\dim\kakko{\Smod\lambda\cdot e\Phi e}
=\rank_{\End(\Smod\lambda\cdot e)}(e\Phi e).
\end{align*}
\qed
\end{thm}

Let $\lambda\vdash nl$ be a partition such that $\len\lambda\le n$
and put $d=K_{\lambda(l^n)}$.
We fix an orthonormal basis
$\{\ve_1^\lambda,\dots,\ve_{f^\lambda}^\lambda\}$
of $\Smod\lambda$ such that the first $d$ vectors
$\ve_1^\lambda,\dots,\ve_d^\lambda$ form a subspace
$(\Smod\lambda)^K$ consisting of $K$-invariant vectors
and left $f^\lambda-d$ vectors form
the orthocomplement of $(\Smod\lambda)^K$
with respect to the $\sym{nl}$-invariant inner product.
The matrix coefficient of $\Smod\lambda$
relative to this basis is
\begin{equation}\label{eq:def_of_psi}
\psi^\lambda_{ij}(g)
=\inprod[\Smod\lambda]{\ve_i^\lambda\cdot g}{\ve_j^\lambda}
\quad(g\in\sym{nl},\ 1\le i,j\le f^\lambda).
\end{equation}
We notice that this function is $K$-biinvariant.
We see that the multiplicity of $\GLmod n\lambda$ in $V\cdot e\Phi e$
is given by the rank of the matrix
\begin{align*}
\kakko{\sum_{h\in H}\varphi(h)\psi^\lambda_{ij}(h)}_{1\le i,j\le d}.
\end{align*}

As a particular case, we obtain the
\begin{thm}\label{thm:mymain}
The multiplicity of the irreducible representation $\GLmod n\lambda$
in the cyclic module $\U\cdot\adet(X)^l$ is equal to the rank of
\begin{equation}
\tramat\lambda(\alpha)
=\kakko{\sum_{h\in H}\anu h\psi^\lambda_{ij}(h)}_{1\le i,j\le d},
\end{equation}
where $\{\psi^\lambda_{ij}\}_{i,j}$ denotes a basis of
the $\lambda$-component of the space $C(K\backslash\sym{nl}/K)$
of $K$-biinvariant functions on $\sym{nl}$
given by \eqref{eq:def_of_psi}.
\end{thm}

\begin{rem}
\begin{enumerate}
\item
By the definition of the basis $\{\psi^\lambda_{ij}\}_{i,j}$
in \eqref{eq:def_of_psi},
we have $\tramat\lambda(0)=I$.
\item
Since $\anu{g^{-1}}=\anu g$ and
$\psi^\lambda_{ij}(g^{-1})=\overline{\psi^\lambda_{ji}(g)}$
for any $g\in\sym{nl}$,
the transition matrices satisfy
$\tramat\lambda(\alpha)^*=\tramat\lambda(\overline\alpha)$.
\item
In Examples \ref{ex:MW2005} and \ref{ex:Jacobi},
the transition matrices
are given by \emph{diagonal matrices}.
We expect that any transition matrix $\tramat\lambda(\alpha)$
is \emph{diagonalizable} in $\Mat_{K_{\lambda(l^n)}}(\C[\alpha])$.
\end{enumerate}
\end{rem}

\begin{ex}[Example \ref{ex:MW2005}]
If $l=1$, then $H=G=\sym n$ and $K=\{1\}$.
Therefore, for any $\lambda\vdash n$, we have
\begin{equation}
\tramat[n,1]\lambda(\phi)
=\frac{n!}{f^\lambda}\inprod[\sym n]{\phi}{\chi^\lambda}I
\end{equation}
by the orthogonality of the matrix coefficients. 
Here $\chi^\lambda$ denotes the irreducible character of $\sym n$
corresponding to $\lambda$.
In particular, if $\phi=\anu\cdot$, then
\begin{equation}
\tramat[n,1]\lambda(\alpha)
=\cp\lambda(\alpha)I
\end{equation}
since the Fourier expansion of $\anu\cdot$ 
(as a class function on $\sym n$) is
\begin{equation}\label{eq:specialFCF}
\anu\cdot
=\sum_{\lambda\vdash n}
\frac{f^\lambda}{n!}f_\lambda(\alpha)\chi^\lambda,
\end{equation}
which is obtained by specializing
the Frobenius character formula for $\sym n$
(see, e.g. \cite{M}).
\end{ex}

The trace of the transition matrix $\tramat\lambda(\alpha)$ is
\begin{equation}
\gcp\lambda{n,l}\deq
\tr\tramat\lambda(\alpha)
=\sum_{h\in H}\anu h\omega^\lambda(h),
\end{equation}
where $\omega^\lambda$ is the \emph{zonal spherical function} for $\lambda$
with respect to $K$ defined by
\begin{align*}
\omega^\lambda(g)\deq
\frac1{\card K}\sum_{k\in K}\chi^\lambda(kg)
\quad(g\in\sym{nl}).
\end{align*}
This is regarded as a generalization of the modified content polynomial
since $\gcp\lambda{n,1}=f^\lambda f_\lambda(\alpha)$ as we see above.
It is much easier to handle these polynomials
than the transition matrices.
If we could prove that a transition matrix $\tramat\lambda$ is a scalar matrix,
then we would have $\tramat\lambda=d^{-1}\gcp\lambda{n,l}I$
and hence we see that the multiplicity of $\GLmod n\lambda$
in $\Cmod nl(\alpha)$ is completely controlled
by the single polynomial $\gcp\lambda{n,l}$. 
In this sense,
it is desirable to obtain a characterization
of the irreducible representations
whose corresponding transition matrices are scalar
as well as to get an explicit expression for the polynomials $\gcp\lambda{n,l}$.
We will investigate these polynomials $\gcp\lambda{n,l}$
and their generalizations in \cite{K2007c}.



\begin{ex}
Let us calculate $\gcp{(nl-1,1)}{n,l}$.
We notice that $\chi^{(nl-1,1)}(g)=\fix_{nl}(g)-1$
where $\fix_{nl}$ denotes the number of fixed points
in the natural action $\sym{nl}\curvearrowright[nl]$.
Hence we see that
\begin{align*}
\gcp{(nl-1,1)}{n,l}
&=\sum_{h\in H}\anu h\frac1{\card K}\sum_{k\in K}(\fix_{nl}(kh)-1)
=\sum_{h\in H}\anu h\frac1{\card K}\sum_{k\in K}\sum_{x\in[nl]}\delta_{khx,x}%
-\sum_{h\in H}\anu h.
\end{align*}
It is easily seen that $khx\ne x$ for any $k\in K$ if $hx\ne x$ ($x\in[nl]$).
Thus it follows that
\begin{align*}
\frac1{\card K}\sum_{k\in K}\sum_{x\in[nl]}\delta_{khx,x}
=\sum_{x\in[nl]}\delta_{hx,x}\frac1{\card K}\sum_{k\in K}\delta_{kx,x}
=\frac1l\fix_{nl}(h)\qquad(h\in H).
\end{align*}
Therefore we have
\begin{align*}
\gcp{(nl-1,1)}{n,l}
&=\frac1l\sum_{h\in H}\anu h\fix_{nl}(h)-\sum_{h\in H}\anu h
=\gcp{(n)}{n,1}^{l-1}\gcp{(n-1,1)}{n,1}\\
&=(n-1)(1-\alpha)(1-(n-1)\alpha)^{l-1}\prod_{i=1}^{n-2}(1+i\alpha)^l.
\end{align*}
We note that the transition matrix $\tramat{(nl-1,1)}$ is a scalar one
(see \cite{KMW}),
so that the multiplicity of $\GLmod n{(nl-1,1)}$ in $\Cmod nl(\alpha)$ is
zero if $\alpha\in\{1,-1,-1/2,\dots,-1/(n-1)\}$ and
$n-1$ otherwise.
\end{ex}

\section{Irreducible decomposition of $\Cmod 2l(\alpha)$ and Jacobi polynomials}

In this section,
as a particular example,
we consider the case where $n=2$ and
calculate the transition matrix $\tramat[2,l]\lambda(\alpha)$ explicitly.
Since the pair $(\sym{2l},K)$ is a \emph{Gelfand pair}
(see, e.g. \cite{M}),
it follows that
\begin{align*}
K_{\lambda(l^2)}
=\inprod[\sym{2l}]{\ind_K^{\sym{2l}}\triv K}{\Smod\lambda}
=1
\end{align*}
for each $\lambda\vdash2n$ with $\len\lambda\le2$.
Thus, in this case,
the transition matrix is just a polynomial and is given by
\begin{equation}
\tramat[2,l]\lambda(\alpha)
=\tr\tramat[2,l]\lambda(\alpha)
=\sum_{h\in H}\anu h\omega^\lambda(h)
=\sum_{s=0}^l\binom ls\omega^\lambda(g_s)\alpha^s.
\end{equation}
Here we put $g_s=(1,l+1)(2,l+2)\dots(s,l+s)\in\sym{2n}$.
Now we write $\lambda=(2l-p,p)$ for some $p$ ($0\le p\le l$).
The value $\omega^{(2l-p,p)}(g_s)$ of the zonal spherical function is
calculated by Bannai and Ito \cite[p.218]{BI} as
\begin{align*}
\omega^{(2l-p,p)}(g_s)
=Q_p(s;-l-1,-l-1,l)
=\sum_{j=0}^p(-1)^j\binom pj\binom{2l-p+1}j\binom lj^{\!\!-2}
\binom sj,
\end{align*}
where
\begin{align*}
Q_n(x;\alpha,\beta,N)&\deq
\tHGF32{-n,n+\alpha+\beta+1,-x}{\alpha+1,-N}{1}\\
&=\sum_{j=0}^N(-1)^j
\binom nj\binom{-n-\alpha-\beta-1}j
\binom{-\alpha-1}j^{\!\!-1}\binom Nj^{\!\!-1}
\binom xj
\end{align*}
is the Hahn polynomial
(see also \cite[p.399]{M}).
We also denote by $\tHGF{n+1}{n}{a_1,\dots,a_{p}}{b_1,\dots,b_{q-1},-N}x$
the hypergeometric polynomial
\begin{align*}
\tHGF pq{a_1,\dots,a_p}{b_1,\dots,b_{q-1},-N}x
=\sum_{j=0}^N\frac{\phs{a_1}j\dots\phs{a_{p}}j}
{\phs{b_1}j\dots\phs{b_{q-1}}j\phs{-N}j}\frac{x^j}{j!}
\end{align*}
for $p,q,N\in\N$ in general (see \cite{AAR}).
Further, if we put
\begin{align*}
G_p^l(x)&\deq
\tHGF21{-p,l-p+1}{-l}{-x}
=\sum_{j=0}^p(-1)^j
\binom pj\binom{l-p+j}j\binom lj^{\!\!-1}x^j,
\end{align*}
then we have the
\begin{thm}\label{thm:my_n=2_case}
Let $l$ be a positive integer.
It holds that
\begin{align*}
\tramat[2,l]{(2l-p,p)}(\alpha)
=\sum_{s=0}^l\binom ls Q_p(s;l-1,l-1,l)\alpha^s
=(1+\alpha)^{l-p}G_p^l(\alpha)
\end{align*}
for $p=0,1,\dots,l$.
\end{thm}

\begin{proof}
Let us put $x=-1/\alpha$.
Then we have
\begin{align*}
\sum_{s=0}^l\binom ls Q_p(s;l-1,l-1,l)\alpha^s
&=\sum_{j=0}^p(-1)^j\binom pj\binom{2l-p+1}j\binom lj^{\!\!-1}
\alpha^j(1+\alpha)^{l-j}\\
&=x^{-l}(x-1)^{l-p}\sum_{j=0}^p\binom pj\binom{2l-p+1}j\binom lj^{\!\!-1}(x-1)^{p-j}
\end{align*}
and
\begin{align*}
(1+\alpha)^{l-p}G_p^l(\alpha)
&=x^{-l}(x-1)^{l-p}\sum_{j=0}^p(-1)^j
\binom pj\binom{l-p+j}j\binom lj^{\!\!-1}(-x)^{p-j}.
\end{align*}
Here we use the elementary identity
\begin{equation*}
\sum_{s=0}^l \binom ls\binom sj\alpha^s=\binom lj\alpha^j(1+\alpha)^{l-j}.
\end{equation*}
Hence, to prove the theorem,
it is enough to verify
\begin{equation}
\sum_{i=0}^p\binom pi\binom{l-p+i}i\binom li^{\!\!-1}x^{p-i}
=\sum_{j=0}^p\binom pj\binom{2l-p+1}j\binom lj^{\!\!-1}(x-1)^{p-j}.
\end{equation}
Comparing the coefficients of Taylor expansion
of these polynomials at $x=1$,
we notice that the proof is reduced to the equality
\begin{equation}
%
\sum_{i=0}^r\binom{l-i}{l-r}\binom{l-p+i}{l-p}
=\binom{2l-p+1}r
\end{equation}
for $0\le r\le p$, which is well known (see, e.g. (5.26) in \cite{GKP}).
Thus we have the conclusion.
\end{proof}
Thus we give another proof of the irreducible decomposition \eqref{eq:decomp_of_V2l}.






\bigskip

\noindent
\textsc{Kazufumi KIMOTO}\\
Department of Mathematical Science,
University of the Ryukyus\\
Senbaru, Nishihara, Okinawa 903-0231, Japan\\
\texttt{kimoto@math.u-ryukyu.ac.jp}


\begin{thebibliography}{9}
\def\article #1: #2.%
{#1: #2.}
\def\paper #1: #2. #3, #4 (#5), #6.%
{\article #1: #2. {\itshape #3} {\bfseries#4} (#5), #6.}
\def\preprint #1: #2.%
{\article #1: #2. Preprint.}
\def\arXiv#1{arXiv:\,\texttt{#1}}
\def\book #1: #2.%
{#1: {\itshape #2.}}

\bibitem{AAR}\book
G.\,E. Andrews, R. Askey and R. Roy:
Special Functions.
Encyclopedia of Mathematics and its Applications, 71.
Cambridge University Press, Cambridge, 1999.

\bibitem{BI}\book
E. Bannai and T. Ito:
Algebraic Combinatorics I, Association Schemes.
The Benjamin/Cummings Publishing Co., Inc., Menlo Park, CA, 1984.

\bibitem{K2007c}\article
K. Kimoto:
Generalized content polynomials toward $\alpha$-determinant cyclic modules.
Preprint (2007).

\bibitem{KMW}\article
K. Kimoto, S. Matsumoto and M. Wakayama:
Alpha-determinant cyclic modules and Jacobi polynomials.
\arXiv{0710.3669}.

\bibitem{KW2007}\paper
K. Kimoto and M. Wakayama:
Invariant theory for singular $\alpha$-determinants.
J. Combin. Theory Ser. A, 115 (2008), no.1, 1--31.

\bibitem{M}\book
I.\,G. Macdonald:
Symmetric Functions and Hall Polynomials, Second edition.
Oxford Univ. Press, 1995.

\bibitem{M2005}\paper
S. Matsumoto:
Alpha-pfaffian, pfaffian point process and shifted Schur measure.
Linear Algebra Appl., 403 (2005), 369--398.

\bibitem{MW2005}\paper
S. Matsumoto and M. Wakayama:
Alpha-determinant cyclic modules of $\mathfrak{gl}_n(\mathbb{C})$.
J. Lie Theory, 16 (2006), 393-405.

\bibitem{GKP}\book
R.\,L. Graham, D.\,E. Knuth and O. Patashnik:
{Concrete Mathematics.
A foundation for computer science. Second edition}.
Addison-Wesley Publishing Company, Reading, MA, 1994.

\bibitem{ST}\paper
T. Shirai and Y. Takahashi:
Random point fields associated with certain Fredholm determinants I:
fermion, Poisson and boson point processes.
J. Funct. Anal., 205 (2003), 414--463.

\bibitem{VereJones}\paper
D. Vere-Jones:
A generalization of permanents and determinants.
Linear Algebra Appl., 111 (1988), 119--124.


\end{thebibliography}
\end{document}